\documentclass[10pt,reqno]{amsart}
\usepackage{amssymb,mathrsfs,graphicx}
\usepackage{ifthen}
\usepackage{hyperref}
\usepackage[margin=1in]{geometry}
\usepackage{caption}
\usepackage{sidecap}
\usepackage{rotating,dsfont}
\usepackage{colortbl}
\definecolor{black}{rgb}{0.0, 0.0, 0.0}
\definecolor{red}{rgb}{1.0, 0.5, 0.5}
\provideboolean{shownotes} 
\setboolean{shownotes}{true} 
\newcommand{\margnote}[1]{
\ifthenelse{\boolean{shownotes}}%
{\marginpar{\raggedright\tiny\texttt{#1}}}%
{}%
}
\newcommand{\hole}[1]{
\ifthenelse{\boolean{shownotes}}%
{\begin{center} \fbox{ \rule {.25cm}{0cm} \rule[-.1cm]{0cm}{.4cm}
\parbox{.85\textwidth}{\begin{center} \texttt{#1}\end{center}} \rule
{.25cm}{0cm}}\end{center}} {} }

\graphicspath{{pics/}}

\title[On the Cauchy problem for the pressureless Euler--Navier--Stokes system]{On the Cauchy problem for the pressureless Euler--Navier--Stokes system in the whole space}

\author[Choi]{Young-Pil Choi}
\address[Young-Pil Choi]{\newline Department of Mathematics\newline
Yonsei University, 50 Yonsei-Ro, Seodaemun-Gu, Seoul 03722, Republic of Korea}
\email{ypchoi@yonsei.ac.kr}

\author[Jung]{Jinwook Jung}
\address[Jinwook Jung]{\newline Research Institute of Basic Sciences \newline Seoul National University, Seoul  08826, Republic of Korea}
\email{warp100@snu.ac.kr}

\numberwithin{equation}{section}

\newtheorem{theorem}{Theorem}[section]
\newtheorem{lemma}{Lemma}[section]
\newtheorem{corollary}{Corollary}[section]
\newtheorem{proposition}{Proposition}[section]
\newtheorem{remark}{Remark}[section]

\newcommand{\R}{\mathbb R}

\newcommand{\N}{\mathbb N}
\newcommand{\ls}{\lesssim}

\newcommand{\T}{\mathbb T}

\newcommand{\mc}{\mathcal C}

\newcommand{\bq}{\begin{equation}}
\newcommand{\eq}{\end{equation}}
\newcommand{\e}{\varepsilon}
\newcommand{\lt}{\left}
\newcommand{\rt}{\right}

\newcommand{\pa}{\partial}

\newcommand{\intr}{\int_{\R^d}}

\begin{document}
\allowdisplaybreaks

\date{\today}

\keywords{Cauchy problem, classical solutions, pressureless Euler equations, incompressible Navier--Stokes equations, large-time behavior.}

\begin{abstract} In this paper, we study the global Cauchy problem for a two-phase fluid model consisting of the pressureless Euler equations and the incompressible Navier--Stokes equations where the coupling of two equations is through the drag force. We establish the global-in-time existence and uniqueness of classical solutions for that system when the initial data are sufficiently small and regular. Main difficulties arise in the absence of pressure in the Euler equations. In order to resolve it, we properly combine the large-time behavior of classical solutions and the bootstrapping argument to construct the global-in-time unique classical solutions.
\end{abstract}

\maketitle \centerline{\date}


%
%
%
%
\section{Introduction}\label{sec:1}
\setcounter{equation}{0}
In the present work, we are interested in the global-in-time existence and uniqueness of classical solutions and its large-time behavior for a coupled hydrodynamic system in the whole space. More precisely, the system of our interest consists of the pressureless Euler equations and incompressible Navier--Stokes equations (in short, pressureless ENS system), which are coupled via the drag force:
\begin{align}
\begin{aligned}\label{A-1}
&\partial_t \rho + \nabla_x \cdot (\rho u) = 0,\quad x \in \R^d, \ t > 0,\\
&\partial_t (\rho u) + \nabla_x \cdot (\rho u \otimes u) = -\rho( u-v),\\
&\partial_t v + (v \cdot\nabla_x )v + \nabla_x p - \mu\Delta_x v = \rho(u-v),\\
&\nabla_x \cdot v =0,
\end{aligned}
\end{align}
subject to initial data:
\bq\label{A-1_ini}
(\rho(x,0), u(x,0), v(x,0)) = (\rho_0(x), u_0(x), v_0(x)), \quad x\in\R^d.
\eq
Here $\rho = \rho(x,t)$ and $u = u(x,t)$ represent the fluid density and velocity for the pressureless flow at a domain $(x,t) \in \R^d \times \R_+$, respectively, and $v = v(x,t)$ denote the fluid velocity for the incompressible flow. $\mu > 0$ is the viscosity coefficient, and throughout this paper, we set $\mu = 1$ for simplicity. 

Our main system is closely related to the kinetic-fluid models, in general multiphase flows, which have received increasing attention due to its wide range of applications, for instances, including medicine, biotechnology, combustion in diesel engines, and atmospheric pollution \cite{BBJM05,BDM03,OR81,VASG06,Wi58}. To be more specific, at the formal level, the pressureless ENS system \eqref{A-1} can be derived from the kinetic-fluid system consisting of Vlasov--Navier--Stokes system with a strong local alignment force. Here we briefly outline the formal derivation. Let $f = f(x,\xi,t)$ be the number density function of dispersed particles in phase space $(x,\xi) \in \R^d \times \R^d$ at time $t \in \R_+$ and $u = u(x,t)$ be the velocity of the incompressible flow. We then consider
\begin{align}\label{kin_fl}
\begin{aligned}
&\pa_t f^\e + \xi \cdot \nabla_x f^\e + \nabla_\xi \cdot \lt((u^\e-\xi)f^\e\rt) = \frac1\e \nabla_\xi \cdot \lt( (\xi-u^\e)f^\e \rt),\cr
&\partial_t v^\e + (v^\e \cdot\nabla_x )v^\e + \nabla_x p^\e -\Delta_x v^\e = \rho^\e(u^\e-v^\e),\\
&\nabla_x \cdot v^\e =0,
\end{aligned}
\end{align}
where $\rho^\e = \rho^\e(x,t)$ and $\rho^\e u^\e = (\rho^\e u^\e)(x,t)$ are local particle density and moment given by
\[
\rho^\e = \intr f\,d\xi \quad \mbox{and} \quad \rho^\e u^\e = \intr \xi f\,d\xi,
\]
respectively. Here the term on the right hand side of the kinetic equation in \eqref{kin_fl} is often called the local alignment force; it produces the dissipation term for the kinetic energy. Moreover, at the formal, it follows from the kinetic equation in \eqref{kin_fl} that $\rho^\e$ and $\rho^\e u^\e$ satisfy
$$\begin{aligned}
&\pa_t \rho^\e + \nabla_x \cdot (\rho^\e u^\e) = 0,\cr
&\pa_t (\rho^\e u^\e) + \nabla_x \cdot \lt(\rho^\e u^\e\otimes u^\e \rt) + \nabla_x \cdot \lt(\intr (\xi - u^\e)\otimes (\xi - u^\e) f\,d\xi \rt) = - \rho^\e(u^\e-v^\e),\cr
&\partial_t v^\e + (v^\e \cdot\nabla_x )v^\e + \nabla_x p^\e -\Delta_x v^\e = \rho^\e(u^\e-v^\e),\\
&\nabla \cdot v^\e =0.
\end{aligned}$$
Even though the above system is not closed, under the strong local alignment regime, i.e., $\e \ll 1$, we have the monokinetic ansatz from \eqref{kin_fl}:
\[
f^\e(x,v,t)\,dxdv \simeq \rho^\e(x,t)\,dx \,\otimes \delta_{u^\e(x,t)}(dv).
\]
This formal observation leads to the pressureless ENS system \eqref{A-1}. There have been some results concerning hydrodynamic limit from the kinetic-fluid models to two-phase fluid systems. In \cite{CCK16,CJ20,CJpre,GJV04, GJV04_2,MV08}, the asymptotic analysis for the Vlasov--Fokker--Planck equation coupled with the incompressible/compressible Navier--Stokes equations are discussed. In these works, the relative entropy method is used, and thus the presence of the diffusion term in the kinetic equation is important. This gives the convexity of the total macroscopic energy and enables us to handle the strong coupling between kinetic and fluid equations. More recently, in \cite{CJpre2}, the hydrodynamic limit for the Vlasov--Poisson--Navier--Stokes equations is investigated and the pressureless Euler--Poisson--Navier--Stokes equations are rigorously derived under the strong local alignment force regime. 

The global existence of classical solutions and its large-time behavior for the pressureless Euler-type system coupled with the incompressible/compressible Navier--Stokes system are investigated in \cite{CK16, HKK14} in the periodic domain $\T^d$. Due to the absence of the pressure term in the Euler system, it is not clear to estimate the fluid density in the desired Sobolev space. In fact, it is well-known that the pressureless Euler-type equations develop a finite-time formation of singularities, for instance $\delta$-shock \cite{CCTT16,CCZ16,CW96,Eng96,HKK14,LT02}. In order to resolve it, the large-time behavior estimate combined with the a priori estimate is used in those works. To be more specific, by introducing a Lyapunov functional, the exponential decay of $L^2$-norm of the pressureless fluid velocity $u$ can be obtained. This together with an appropriate higher-order $L^2$-Sobolev norm of $u$ asserts the exponential decay of $\|\nabla_x u\|_{H^s(\T^d)}$ in time. From this, we can bound the fluid density $\rho$ from both above and below by some positive constant, which is independent of time. This requires a rather different regularity for $\rho$ and $u$ for instance, $\rho \in \mc([0,+\infty); H^s(\T^d))$ and $u \in \mc([0,+\infty); H^{s+2}(\T^d))$ for some $s > 0$. However, in this strategy, the Poincar\'e inequality is crucially used, and thus it is not clear to extend this idea to Cauchy problems in the whole space. On the other hand, for the coupled isothermal Euler and incompressible/compressible Navier--Stokes system, the large-time behavior estimate is not necessarily required for the global-in-time regularity \cite{Choi15,Choi16}.

The main purpose of the current work is to develop a global existence theory for the pressureless ENS system \eqref{A-1}. We employ similar ideas to that for the periodic domain case \cite{CK16, HKK14} to construct the global-in-time classical solutions. We use a recent work \cite{Hpre}, where the large-time behavior for the Vlasov--Navier--Stokes system in the whole space is discussed, to have the $L^2$-decay estimate of solutions in the whole space. In particular, this shows a polynomial decay of kinetic energies for each system in \eqref{A-1}. This combined with our careful analysis enables us to have the decay estimate in higher-order Sobolev spaces; we obtain the same decay rate as the lower order estimate. Here, a proper combination of the drag forcing effect and the smoothing effect from the viscosity  in the Navier--Stokes system is significantly used. Although we can not expect the exponential decay for our system in the whole space but only the polynomial decay rate, it is enough to have the uniform-in-time bound estimates for the fluid density. Combining these estimates and the standard bootstrapping argument gives the global-in-time existence and the large-time behavior of classical solutions to the system \eqref{A-1}.

More precisely, we state our main theorem. 

\begin{theorem}\label{T2.1}
Let $d \ge 3$ and $s \geq 2[d/2]\footnote{$[\,\cdot\,]$ represents the floor function, i.e., $[x]$ is the greatest integer less than or equal to $x \in \R$.}+1$. Suppose that the initial data $(\rho_0, u_0, v_0)$ satisfy
\begin{itemize}
\item[(i)] $\rho_0(x)>0$ for every $x\in\R^d$ and 
\item[(ii)] $(\rho_0, u_0, v_0) \in (L^1 \cap H^s)(\R^d) \times H^{s+2}(\R^d) \cap (L^1 \cap H^{s+1})(\R^d)$.
\end{itemize}
If 
\[
\|\rho_0\|_{H^s(\R^d)} + \|u_0\|_{H^{s+2}(\R^d)} + \|v_0\|_{H^{s+1}(\R^d)} + \|v_0\|_{L^1(\R^d)}<\e_0
\]
for $\e_0$ sufficiently small, the Cauchy problem \eqref{A-1}-\eqref{A-1_ini} has a unique global classical solution $(\rho, u, v) \in \mc([0,+\infty);H^s(\R^d)) \times \mathcal{C}([0,+\infty);H^{s+2}(\R^d)) \times \mathcal{C}([0,+\infty);H^{s+1}(\R^d))$ satisfying $\rho(x,t) > 0$ for all $(x,t) \in \R^d \times [0, + \infty)$ and
\[
\sup_{t \ge 0}\lt( \|\rho(\cdot,t)\|_{H^s(\R^d)} + \|u(\cdot,t)\|_{H^{s+2}(\R^d)} + \|v(\cdot,t)\|_{H^{s+1}(\R^d)} \rt) <\infty.
\]
Moreover, for every $\alpha \in (0,d/2)$ there exists a constant $C>0$ independent of $t$ such that
\[
\|u(\cdot,t)\|_{H^{s+1}(\R^d)}^2 + \|v(\cdot,t)\|_{H^s(\R^d)}^2 \le \frac{C}{(1+t)^\alpha}\quad \forall \, t \ge 0.
\]
\end{theorem}
\begin{remark} Due to the absence of the pressure in Euler equation in \eqref{A-1}, we can not have time decay estimate of the fluid density $\rho$.
\end{remark}

The rest of this paper is organized as follows. In Section \ref{sec:pre}, we briefly present some useful Sobolev inequalities and a priori energy estimates for the system \eqref{A-1}. We also state the local-in-time existence theory which can be established by the standard arguments developed for conservation laws types. Section \ref{sec:ap} is devoted to present the time decay estimate of the total energy for the system \eqref{A-1}. As mentioned before, we use the strategy recently developed in \cite{Hpre}. We slightly modify the time behavior estimate to apply it to our main system \eqref{A-1}. Finally, in Section \ref{sec:main}, we provide the a priori estimates of solutions in the weighted Sobolev spaces by $(1+t)^r$ with some $r > 0$. These estimates yield that the local-in-time classical solutions can be extended to the global ones. Using these global solutions, we refine the weighted Sobolev space estimates to establish the large-time behavior estimates with the desired polynomial decay rate. This proves Theorem \ref{T2.1}. 

Before closing this section, we introduce several notations used throughout this paper. For a function $f=f(x)$, $\|f\|_{L^p}$ represents the usual $L^p(\R^d)$-norm. For simplicity, we omit $x$-dependence of differential operators, i.e., $\nabla f:= \nabla_x f$ and $\Delta f := \Delta_x f$. We denote by $C$ a generic positive constant and it may differ from line to line. Finally, $f \ls g$ represents that there exists a positive constant $C>0$ such that $f \leq Cg$.

%
%
%
\section{Preliminaries}\label{sec:pre}
In this section, we provide useful Sobolev inequalities and conservation laws for the pressureless ENS system \eqref{A-1} that will be significantly used later. We also state the local-in-time existence and uniqueness theorem.

We first recall the Moser-type inequalities.
\begin{lemma}\label{lem_moser} 
\begin{itemize}
\item[(i)] For any pair of functions $f,g \in (H^k \cap L^\infty)(\R^d)$, we obtain
\[
\|\nabla^k (fg)\|_{L^2} \le C\lt(\|f\|_{L^\infty} \|\nabla^k g\|_{L^2} + \|\nabla^k f\|_{L^2}\|g\|_{L^\infty}\rt).
\]
Furthermore, if $\nabla f \in L^\infty(\R^d)$, we have
\[
\|\nabla^k(fg) - f\nabla^k g\|_{L^2} \le C\lt(\|\nabla f\|_{L^\infty}\|\nabla^{k-1} g\|_{L^2} + \|g\|_{L^\infty}\|\nabla^k f\|_{L^2}\rt).
\]
Here $C>0$ only depends on $k$ and $d$.

\item[(ii)] For $f \in H^{[d/2]+1}(\R^d)$, we have
\[
\|f\|_{L^\infty} \leq C\|\nabla f\|_{H^{[d/2]}}.
\]
\end{itemize}
\end{lemma}

We next provide estimates of the mass and the total momentum, and the energy dissipation of the system \eqref{A-1}. Since the proof is almost the same as \cite[Lemma 2.1]{HKK14}, see also \cite{Choi15,Choi16,Choi16_2}, we omit it here.
\begin{lemma}\label{lem_energy} Let $(\rho,u,v)$ be a solution to the system \eqref{A-1} with sufficient integrability. Then we have
\begin{itemize}
\item[(i)] The total mass of $\rho$ is conserved in time:
\[
\intr \rho(x,t)\,dx = \intr \rho_0(x)\,dx \quad \forall \, t\geq0.
\]
\item[(ii)] The total momentum is conserved in time:
\[
\intr (\rho u)(x,t)\,dx + \intr v(x,t)\,dx = \intr (\rho_0 u_0)(x)\,dx + \intr v_0(x)\,dx\quad \forall \, t\geq0.
\]
\item[(iii)] The total energy is not increasing in time:
$$\begin{aligned}
&\frac12 \intr \rho(x,t)|u(x,t)|^2\,dx + \frac12\intr |v(x,t)|^2\,dx \cr
&\quad + \int_0^t\lt( \intr |\nabla v(x,\tau)|^2\,dx + \intr \rho(x,\tau)|(u-v)(x,\tau)|^2\,dx\rt)d\tau \cr
&\qquad = \frac12 \intr \rho_0(x)|u_0(x)|^2\,dx + \frac12\intr |v_0(x)|^2\,dx\quad \forall \, t\geq0.
\end{aligned}$$
\end{itemize}
\end{lemma}
For the sake of notational simplicity, we set a total energy $E$ and its dissipation rate $D$:
\[
E(t):= \intr \rho(x,t)|u(x,t)|^2\,dx + \frac12\intr |v(x,t)|^2\,dx
\]
and
\[
D(t):= \intr |\nabla v(x,t)|^2\,dx + \intr \rho(x,t)|(u-v)(x,t)|^2\,dx,
\]
respectively.

Finally, in the theorem below, we present the local-in-time existence and uniqueness of classical solutions to our main system \eqref{A-1}. 
\begin{theorem}\label{thm_local} Let $d \ge 3$ and $s \geq 2[d/2]+1$. Suppose that the initial data $(\rho_0, u_0, v_0)$ satisfy the assumptions (i) and (ii) in Theorem \ref{T2.1}. Then for any positive constants $\epsilon_0 < \delta_0$, there exists a positive constant $T_0$ depending only on $\epsilon_0$ and $\delta_0$ such that if 
\[
\|\rho_0\|_{H^s} + \|u_0\|_{H^{s+2}} + \|v_0\|_{H^{s+1}} + \|v_0\|_{L^1}<\epsilon_0,
\]
then the pressureless ENS system \eqref{A-1}-\eqref{A-1_ini} admits a unique solution 
\[
(\rho, u, v) \in \mc([0,T_0];H^s(\R^d)) \times \mathcal{C}([0,T_0];H^{s+2}(\R^d)) \times \mathcal{C}([0,T_0];H^{s+1}(\R^d))
\] 
satisfying $\rho(x,t) > 0$ for all $(x,t) \in \R^d \times [0,T_0]$ and
\[
\sup_{0 \leq t \leq T_0}\lt( \|\rho(\cdot,t)\|_{H^s} + \|u(\cdot,t)\|_{H^{s+2}} + \|v(\cdot,t)\|_{H^{s+1}} \rt) \leq \delta_0.
\]
\end{theorem}
The study of local-in-time existence theory for each equation in \eqref{A-1} is by now well-developed in the $H^s$ Sobolev space. Thus we skip the proof of the above theorem. We refer to \cite{CK16,HKK14} for the readers who are interested in it. 

%
%
%

\section{A priori estimate of the large-time behavior of solutions}\label{sec:ap}
\setcounter{equation}{0}
In this section, we present a priori estimate of the time behavior of solutions to the pressureless ENS system \eqref{A-1}. For this, we first investigate the large-time behavior of solutions to the heat equation.
\begin{lemma}\label{L3.1}
Let $V$ be a solution to the heat equation on $\R^d$:
\bq\label{heat_eq}
\pa_t V -\Delta V = 0 \quad \mbox{with} \quad V|_{t=0} = v_0.
\eq
Then, we have
\[
\|V(\cdot,t)\|_{L^2}^2 \le \frac{C(\|v_0\|_{L^2}^2 + \|v_0\|_{L^1}^2)}{(1+t)^{d/2}} \quad \forall \,t \ge0,
\]
where $C$ is independent of $v_0$ and $t$.
\end{lemma}
\begin{proof} By taking the Fourier transform of \eqref{heat_eq}, we find
\[
\pa_t \widehat V(\xi,t) + |\xi|^2 \widehat V(\xi,t) = 0, 
\]
and subsequently solving the above differential equation implies
\[
\widehat V(\xi,t) = \widehat V_0(\xi) e^{-|\xi|^2 t}.
\]
We next estimate $L^2$-norm of $\widehat V$. Based on the proof in \cite{Sch85}, we set
\[
{\bf r}(t) :=\lt(\frac{d}{2(1+t)}\rt)^{1/2}. 
\]
Then we obtain
\begin{align}\label{est_v2}
\begin{aligned}
\frac12\frac{d}{dt}\|\widehat V\|_{L^2}^2 &= -\intr |\xi|^2 |\widehat V|^2\,d\xi\\
&= - \int_{\{\xi\in\R^d:|\xi|\ge {\bf r}(t)\}} |\xi|^2 |\widehat V|^2\,d\xi -  \int_{\{\xi\in\R^d:|\xi|\le {\bf r}(t)\}} |\xi|^2 |\widehat V|^2\,d\xi\\
&\le -\frac{d}{2(1+t)}\int_{\{\xi\in\R^d:|\xi|\ge {\bf r}(t)\}} |\widehat V|^2\,d\xi -  \int_{\{\xi\in\R^d:|\xi|\le {\bf r}(t)\}} |\xi|^2 |\widehat V|^2\,d\xi\\
&= -\frac{d}{2(1+t)}\|\widehat V\|_{L^2}^2 +   \int_{\{\xi\in\R^d:|\xi|\le {\bf r}(t)\}}\lt(\frac{d}{2(1+t)} -|\xi|^2\rt) |\widehat V|^2\,d\xi.
\end{aligned}
\end{align}
Since 
\[
|\widehat V(\xi,t)| = |\widehat V_0(\xi)|e^{-|\xi|^2 t} \le \|v_0\|_{L^1}e^{-|\xi|^2 t},
\] 
this yields
\[\begin{aligned}
\int_{\{\xi\in\R^d:|\xi|\le {\bf r}(t)\}}\lt(\frac{d}{2(1+t)} -|\xi|^2\rt) |\widehat V|^2\,d\xi
& \leq \frac{d}{2(1+t)}\int_{\{\xi\in\R^d:|\xi|\le {\bf r}(t)\}} |\widehat V|^2\,d\xi \cr
&\le \frac{C\|v_0\|_{L^1}^2}{2(1+t)}\int_0^{{\bf r}(t)} \tau^{d-1} e^{-2\tau^2 t}\, d\tau \cr
&\le \frac{C\|v_0\|_{L^1}^2}{2(1+t)} ({\bf r}(t))^d \cr
&\le \frac{C\|v_0\|_{L^1}^2}{(1+t)^{d/2+1}},
\end{aligned}\]
where $C$ is independent of $t$ and $V$. Combining this with \eqref{est_v2} gives
\[
\frac{d}{dt}\lt( (1+t)^d \|\widehat V(\cdot,t)\|_{L^2}^2\rt) \le C\|v_0\|_{L^1}^2 (1+t)^{d/2-1}.
\]
We then integrate it over the time interval $[0,t]$ to have 
\[
\|\widehat V(\cdot,t)\|_{L^2}^2 \le \frac{\|\widehat V_0\|_{L^2}^2}{(1+t)^d} + C\|v_0\|_{L^1}^2\frac{(1+t)^{d/2}-1}{(1+t)^d} \leq  \frac{C\lt(\|\widehat V_0\|_{L^2}^2 + \|v_0\|_{L^1}^2\rt)}{(1+t)^{d/2}}.
\]
Finally, we use Plancherel's Theorem to conclude the desired result.
\end{proof}

Next, we study a priori estimate  for the large-time behavior of solutions to the system \eqref{A-1}.

\begin{proposition}\label{P3.1}
For $T>0$ and $d\ge 3$, let $(\rho,u,v)$ be a classical solution to the pressureless ENS system \eqref{A-1} on the time interval $[0,T]$ satisfying $\|\rho\|_{L^\infty(\R^d \times (0,T))}<\infty$. Then, there exists a constant $C>0$ independent of $T$ such that for every $\alpha \in (0,d/2)$,
\bq\label{large_time}
E(t)(1+t)^\alpha + \int_0^t (1+\tau)^\alpha D(\tau)\,d\tau \le C(E(0)+\|v_0\|_{L^1}^2) \quad \forall \, t \in [0,T].
\eq
\end{proposition}
\begin{proof}
Since the proof is almost the same as \cite[Theorem 2.1]{Hpre}, we only provide the sketch of proof. We notice from Lemma \ref{lem_energy} that
\bq\label{P3-1.1}
\frac{d}{dt}E(t) + D(t) = 0.
\eq
We then estimate the lower bound of the dissipation term $D(t)$. Introducing a continuous cutoff $g(t)$, which is bounded uniformly in $t$ and will be specified later, we estimate 
$$\begin{aligned}
\intr |\nabla v|^2\,dx &= \intr |\xi|^2 |\widehat v|^2\,d\xi \cr
&\ge \int_{\{\xi\in\R^d:|\xi| \ge g(t)\}}   |\xi|^2 |\widehat v|^2\,d\xi \cr
&\ge g^2(t)\intr |v|^2\,dx -g^2(t) \int_{\{\xi\in\R^d:|\xi|\le g(t)\}} |\widehat v|^2\,d\xi,
\end{aligned}$$
where we used Plancherel's Theorem.
Moreover, we obtain
\[
\intr \rho|u-v|^2\,dx \ge \frac12 \intr \rho|u|^2\,dx - \|\rho\|_{L^\infty(\R^d \times (0,T))}\|v\|_{L^2}^2.
\]
We next choose a constant $C_0>0$ such that
\[
\frac{\|\rho\|_{L^\infty}}{1+C_0} \leq \frac12  \quad \mbox{and} \quad \frac{\|g\|_{L^\infty}^2}{1+C_0} \le \frac12.
\]
This gives
\[
D(t) \ge \frac12 D(t)  + \tilde{g}(t)^2E(t) - \frac{g^2(t)}{2} \int_{\{\xi\in\R^d:|\xi|\le g(t)\}}|\widehat v|^2\,d\xi,
\]
where $\tilde{g}(t)^2 :=  g^2(t)/(4(1+C_0))$. Combining this with \eqref{P3-1.1}, we obtain
\bq\label{P3-1.2}
E(t) + \int_s^t \tilde{g}^2(\tau)E(\tau)\,d\tau +\frac12\lt(\int_s^t D(\tau)\,d\tau\rt) \le E(s) + \frac12\int_s^t g^2(\tau)\int_{\{\xi\in\R^d:|\xi|\le g(\tau)\}}|\widehat v(\xi,\tau)|^2\,d\xi d\tau
\eq
for $0\le s \le t \le T$. Now, we set $F := \rho(u-v)$, and let $V$ be the solution of the heat equation \eqref{heat_eq} corresponding to the initial data $V(x,0) = v_0(x)$. Then, applying Duhamel's formula to the incompressible Navier--Stokes equations in \eqref{A-1} yields
\[
\widehat v(\xi,\tau) = \widehat V(\xi,\tau) + \int_0^\tau (-\widehat{v \cdot \nabla v} - \widehat{\nabla p} + \widehat{F})(\xi,r)e^{|\xi|^2(r-\tau)}\,dr.
\]
On the other hand, the incompressibility condition implies
\[
|\widehat{\nabla p}| = \lt| \frac{\xi \cdot (-\widehat{v \cdot \nabla v} + \widehat{F})}{|\xi|^2} \xi \rt| \le \lt|-\widehat{v \cdot \nabla v} + \widehat{F}\rt|,
\]
and thus we get
$$\begin{aligned}
&\int_{\lt\{\xi\in\R^d:|\xi|\le g(\tau)\rt\}} |\widehat{v}(\xi,\tau)|^2\,d\xi \cr
&\quad \le C\lt(\|V(\cdot,\tau)\|_{L^2}^2 + g(\tau)^{d+2} \lt(\int_0^\tau \|v(\cdot,r)\|_{L^2}^2\,dr\rt)^2 + g(\tau)^d \lt( \int_0^\tau \|\rho(u-v)(\cdot,r)\|_{L^1}\,dr\rt)^2 \rt),
\end{aligned}$$
where $C>0$ is independent of $T$. This together with \eqref{P3-1.2} gives
$$
\begin{aligned}
&E(t) + \int_s^t \tilde{g}^2(\tau)E(\tau)\,d\tau +\frac12\lt(\int_s^t D(\tau)\,d\tau\rt)\\
&\quad \le E(s) + C\int_s^t g^2(\tau)\|V(\cdot,\tau)\|_{L^2}^2\,d\tau + C\int_s^t g(\tau)^{d+4} \lt(\int_0^\tau \|v(\cdot,r)\|_{L^2}^2\,dr\rt)^2\,d\tau\\
&\qquad + C\int_s^t g^{d+2}(\tau)\lt(\int_0^\tau \|\rho(u-v)(\cdot,r)\|_{L^1}\,dr\rt)^2\,d\tau.
\end{aligned}
$$
Since
\[
\|\rho(u-v)\|_{L^1} \le \|\rho\|_{L^1}^{1/2} \lt(\intr \rho |u-v|^2\,dx\rt)^{1/2} = \lt(\intr \rho |u-v|^2\,dx\rt)^{1/2},
\]
we use Gr\"onwall-type lemma in \cite[Lemma 2.2]{Hpre} to deduce
\begin{align}\label{est_decay}
\begin{aligned}
E(t)& \exp\lt(\int_0^t \tilde{g}^2(\tau)\,d\tau\rt) +\frac12\int_0^t D(\tau) \exp\lt(\int_0^\tau \tilde{g}^2(r)\,dr\rt)\,d\tau\\
&\le E(0) + C\int_0^t g^2(\tau)\|V(\cdot,\tau)\|_{L^2}^2  \exp\lt(\int_0^\tau \tilde{g}^2(r)\,dr\rt)\,d\tau\\
&\quad + C\int_0^t g^{d+4}(\tau) \lt(\int_0^\tau \|v(\cdot,r)\|_{L^2}^2\,dr\rt)^2  \exp\lt(\int_0^\tau \tilde{g}^2(r)\,dr\rt)\,d\tau\\
&\quad + C\int_0^t g^{d+2}(\tau)\lt( \int_0^\tau D^{1/2}(r)\,dr\rt)^2  \exp\lt(\int_0^\tau \tilde{g}^2(r)\,dr\rt)\,d\tau.
\end{aligned}
\end{align}
Now, we further choose $g$ satisfying
\[
g^2(t) = \frac{4\alpha (1+C_0)}{1+t}, \quad \mbox{i.e.,} \quad \tilde{g}^2(t) = \frac{\alpha}{10+t},
\]
where $\alpha \in [1, d/2)$ to be determined later. This gives
\[
\exp\lt(\int_0^\tau \tilde{g}^2(r)\,dr\rt) = (10+\tau)^\alpha,
\]
and we use Lemma \ref{L3.1} to get
$$
\begin{aligned}
E(0) &+ C\int_0^t g^2(\tau) \|V(\cdot,\tau)\|_{L^2}^2 \exp\lt(\int_0^\tau \tilde{g}^2(r)\,dr\rt)\,d\tau\\
&\le C\lt( E(0) + (\|v_0\|_{L^2} + \|v_0\|_{L^1}^2)\int_0^t (1+\tau)^{-(1+d/2)+\alpha}\,d\tau\rt) \\
&\le C\lt(E(0) + \|v_0\|_{L^1}^2\rt).
\end{aligned}
$$
Here $C > 0$ is independent of $T$. To get the desired result, we prove some a priori estimates. First, assume that for some $\beta\in[0,d/2)$,
\[
E(t) \le \frac{C(E(0)+\|v_0\|_{L^1}^2)}{(1+t)^\beta} \quad \forall\,t \in [0,T]. 
\]
Then, under this a priori assumption, we estimate the third term on the right hand side of the inequality \eqref{est_decay} as
\bq\label{P3-1.5}
\begin{aligned}
\int_0^t& g^{d+4}(\tau) \lt(\int_0^\tau \|v(\cdot,r)\|_{L^2}^2\,dr\rt)^2  \exp\lt(\int_0^\tau \tilde{g}^2(r)\,dr\rt)\,d\tau\\
&\le C(E(0)+\|v_0\|_{L^1}^2) \int_0^t (1+\tau)^{\alpha-2\beta -d/2}\,d\tau\\
&\le C(E(0)+\|v_0\|_{L^1}^2)\times \left\{\begin{array}{lc} (1+t)^{\alpha-2\beta-d/2+1} & \mbox{if} \quad\alpha-2\beta -d/2 >-1, \\
1 & \mbox{if} \quad \alpha-2\beta -d/2 <-1.\end{array}\right.
\end{aligned}
\eq
Moreover, if we assume that 
\[
\int_0^\tau D(\tau) (10+\tau)^\alpha \le C(E(0)+\|v_0\|_{L^1}^2)\frac{(10+t)^\alpha}{(1+t)^\beta} \quad \forall\,t \in [0,T],
\]
then we have
\bq\label{P3-1.6}
\begin{aligned}
\int_0^t &g^{d+2}(\tau)\lt( \int_0^\tau D^{1/2}(r)\,dr\rt)^2  \exp\lt(\int_0^\tau \tilde{g}^2(r)\,dr\rt)\,d\tau\\
&\le C\int_0^t (1+\tau)^{\alpha-d/2-1} \lt[\int_0^\tau D(r)(10+r)^\alpha\,dr\rt]\lt[\int_0^\tau (1+r)^{-\alpha}\,dr\rt]  \,d\tau\\
&\le C(E(0)+\|v_0\|_{L^1}^2) \int_0^t (1+\tau)^{2\alpha-\beta -d/2-1}\,d\tau\\
&\le C(E(0)+\|v_0\|_{L^1}^2)\times \left\{\begin{array}{lc} (1+t)^{2\alpha-\beta-d/2} & \mbox{if} \quad 2\alpha-\beta -d/2 > 0, \\
1 & \mbox{if} \quad 2\alpha-\beta -d/2 <0.\end{array}\right.
\end{aligned}
\eq
Note that \eqref{P3-1.1} implies that \eqref{P3-1.5} and \eqref{P3-1.6} actually hold for $\alpha=1$ and $\beta=0$. From now on, we first let $\alpha=1$ and $\beta=0$ in \eqref{P3-1.5}, \eqref{P3-1.6} and follow the procedure based on the inductive argument in \cite[Theorem 2.1]{Hpre}. We can construct sequences $\alpha_n$ and $\beta_n$ such that \eqref{P3-1.5} and \eqref{P3-1.6} hold with $\alpha=\alpha_n$ and $\beta=\beta_n$ for each $n \in \N$, and $\alpha_n \to d/2$ and $\beta_n \to (d/2)(1-\e)/(1+\e)$ for any $\e>0$ as $n \to \infty$. This completes the proof.
\end{proof}

%
%
%

\section{Proof of Theorem \ref{T2.1}}\label{sec:main}

\subsection{A priori estimates}
In this part, we provide the a priori estimates for the global-in-time existence of classical solutions. Let $T>0$, $d \ge 3$, and $s \geq 2[d/2]+1$. Throughout this subsection, we assume that for a sufficiently small $\e_1>0$,
\[
\mathfrak{X}(s;T) := \sup_{0 \le t \le T}\lt( \|\rho(\cdot,t)\|_{H^s}^2 + \|u(\cdot,t)\|_{H^{s+2}}^2 + \|v(\cdot,t)\|_{H^{s+1}}^2\rt) \le \e_1^2 \ll 1.
\]
We denote by
\[
\mathfrak{X}_0(s) := \|\rho_0\|_{H^s}^2 + \|u_0\|_{H^{s+2}}^2 + \|v_0\|_{H^{s+1}}^2.
\]
Our main goal of this subsection is to prove the following uniform-in-time estimate.
\begin{proposition}\label{P3.2}
For $T>0$, suppose that $\e_1>0$ is sufficiently small satisfying
\[
\mathfrak{X}(s;T) + \|v_0\|_{L^1}^2 \le \e_1^2.
\] 
Then, there exists a constant $C^*>0$ independent of $T$ such that
\[
\mathfrak{X}(s;T) \le C^*\lt(\mathfrak{X}_0(s) + \|v_0\|_{L^1}^2\rt).
\]
\end{proposition}
We provide uniform-in-time estimates in the following order:
\begin{enumerate}
\item[(i)] $L^\infty(0,T;L^2(\R^d))$ estimate of $u$ (Lemma \ref{L3.2}).
\vspace{0.2cm}

\item[(ii)] $L^2(\R^d \times (0,T))$ estimate of $(1+t)^r\nabla^k u$ for $1 \le k\le s+1$ and  $r \in (0,3/4)$ (Lemma \ref{L3.3}).
 \vspace{0.2cm}
 
\item[(iii)] $L^\infty(0,T;L^2(\R^d))$ estimate of $(1+t)^r\nabla^k v$ and $L^2(\R^d \times (0,T))$ estimate of $(1+t)^r \nabla^{k+1} v$ for $1 \le k\le s$ and $r \in (0,3/4)$ (Lemma \ref{L3.4}).
\vspace{0.2cm}

\item[(iv)] $L^\infty(0,T;L^2(\R^d))$ estimates of $\nabla^{s+1} v$ and $\nabla^{s+2} u$ (Lemmas \ref{L3.5} and \ref{L3.6}).
\vspace{0.2cm}

\item[(v)] $L^\infty(0,T;H^s(\R^d))$ estimates of $\rho$ (Lemma \ref{C3.3}).
\end{enumerate}

We proceed to the first step.
\begin{lemma}\label{L3.2}
There exists a constant $C>0$ independent of $T$ such that
\[
\|u(\cdot,t)\|_{L^2}^2 + \int_0^t \|u(\cdot,\tau)\|_{L^2}^2\,d\tau \le C\lt(\mathfrak{X}_0(s) + \|v_0\|_{L^1}^2\rt).
\]
\end{lemma}
\begin{proof}
Direct computation gives
\[
\begin{aligned}
\frac12 \frac{d}{dt}\|u\|_{L^2}^2 + \|u\|_{L^2}^2 &= \frac12 \intr (\nabla \cdot u)|u|^2\,dx + \intr u \cdot v\,dx\\
&\le C\|\nabla u\|_{L^\infty}\|u\|_{L^2}^2 + \|u\|_{L^2}\|v\|_{L^2}\\
&\le  \lt(\frac14 +C\|\nabla u\|_{L^\infty}\rt)\|u\|_{L^2}^2  + \|v\|_{L^2}^2\\
&\le \lt(\frac14 +C\e_1\rt)\|u\|_{L^2}^2  + \|v\|_{L^2}^2.
\end{aligned}
\]
We then choese $\e_1$ sufficiently small so that $1/4 + C\e_1 < \frac12$ to get
\[
\|u(\cdot,t)\|_{L^2}^2 + \int_0^t \|u(\cdot,\tau)\|_{L^2}^2 \,d\tau\le \|u_0\|_{L^2}^2 +  2\int_0^t \|v(\cdot,\tau)\|_{L^2}^2\,d\tau.
\]
On the other hand, Proposition \ref{P3.1} with the choice $\alpha>1$ implies
\[
\int_0^t \|v(\cdot,\tau)\|_{L^2}^2\,d\tau \le C(E(0)+\|v_0\|_{L^1}^2) \int_0^t (1+\tau)^{-\alpha}\,d\tau \le C(E(0)+\|v_0\|_{L^1}^2).
\]
This together with the fact $E(0) \le C\mathfrak{X}_0(s)$ concludes the desired result.
\end{proof}

\begin{lemma}\label{L3.3} 
For $r \in (0,3/4)$ and $1\le k \le s+1$, there exists a constant $C>0$ independent of $T$ such that
\[
\|(1+t)^r \nabla^k u(\cdot,t)\|_{L^2}^2 + \int_0^t \|(1+\tau)^r \nabla^k u(\cdot,\tau)\|_{L^2}^2\,d\tau \le C\lt(\|\nabla^k u_0\|_{L^2}^2+ \int_0^t \|(1+\tau)^r \nabla^k v(\cdot,\tau)\|_{L^2}^2\,d\tau\rt).
\]
\end{lemma}
\begin{proof}
From the momentum equation in the pressureless Euler equation in \eqref{A-1}, we can get
\[
\pa_t \lt((1+t)^r \nabla^k u\rt) + (1+t)^r \nabla^k (u \cdot \nabla u) = (1+t)^r \nabla^k (v-u) + r(1+t)^{r-1} \nabla^k u.
\]
With $r+1/8 \in (1/8, 7/8)$ in mind,  we use Young's inequality to obtain
$$\begin{aligned}
&\frac12\frac{d}{dt}\|(1+t)^r \nabla^k u\|_{L^2}^2 \cr
&\quad = -(1+t)^{2r}\intr \lt(u \cdot \nabla (\nabla^k u)\rt) \nabla^k u\,dx  - (1+t)^{2r}\intr \lt(\nabla^k (u \cdot \nabla u) - u \cdot \nabla (\nabla^k u)\rt)  \nabla^k u\,dx\\
&\qquad -(1+t)^{2r}\|\nabla^k u\|_{L^2}^2 + (1+t)^{2r} \intr \nabla^k u  \nabla^k v\,dx + r(1+t)^{2r-1}\|\nabla^k u\|_{L^2}^2\\
&\quad \le C\|\nabla u\|_{L^\infty}\|(1+t)^r \nabla^k u\|_{L^2}^2 -(1+t)^{2r}\|\nabla^k u\|_{L^2}^2 + (7/8-r)(1+t)^{2r}\|\nabla^k u\|_{L^2}^2\\
&\qquad + C\|(1+t)^r \nabla^k v\|_{L^2}^2 + r(1+t)^{2r-1}\|\nabla^k u\|_{L^2}^2\\
&\quad\le C\|\nabla u\|_{L^\infty}\|(1+t)^r \nabla^k u\|_{L^2}^2 - \lt(\frac{(r+1/8)(1+t) -r}{1+t}\rt)\|(1+t)^r\nabla^k u\|_{L^2}^2 + C\|(1+t)^r \nabla^k v\|_{L^2}^2\\
&\quad\le -(1/8 - C\e_1)\|(1+t)^r \nabla^k u\|_{L^2}^2+ C\|(1+t)^r \nabla^k v\|_{L^2}^2,
\end{aligned}$$
where $C>0$ is independent of $T$. Here, we can choose $\e_1$ sufficiently small so that $1/8 - C\e_1 >1/16$. Thus, we integrate the previous relation with respect to $t$ to get the desired result.
\end{proof}
 
Next, we get the uniform-in-time estimate for $v$.

\begin{lemma}\label{L3.4}
For $r \in (0,3/4)$ and $1\le k \le s$, there exists a constant $C>0$ independent of $T$ such that 
\[
\begin{aligned}
\|(&1+t)^r \nabla^k v(\cdot,t)\|_{L^2}^2 + \int_0^t \|(1+\tau)^r \nabla^{k+1} v(\cdot,\tau)\|_{L^2}^2\,d\tau \\
& \le \|\nabla^k v_0\|_{L^2}^2 + C\sum_{\ell=1}^k \lt(\int_0^t \|(1+\tau)^r \nabla^\ell u(\cdot,\tau)\|_{L^2}^2\,d\tau + \int_0^t \|(1+\tau)^r \nabla^\ell v(\cdot,\tau)\|_{L^2}^2\,d\tau\rt).
\end{aligned}
\]
\end{lemma}
\begin{proof}
We deduce from the Navier--Stokes equations in \eqref{A-1} that
\begin{align*}
\frac12&\frac{d}{dt}\|(1+t)^r \nabla^k v\|_{L^2}^2 + \|(1+t)^r \nabla^{k+1}v\|_{L^2}^2\\
&= -(1+t)^{2r} \intr \lt(\nabla^k (v \cdot \nabla v) - v \cdot \nabla (\nabla^k v)\rt)  \nabla^k v\,dx +r(1+t)^{2r-1}\|\nabla^k v\|_{L^2}^2\\
&\quad - (1+t)^{2r}\intr \nabla^{k-1}(\rho(u-v)) \nabla^{k+1} v\,dx\\
&\le C\|\nabla v\|_{L^\infty} \|(1+t)^r \nabla^k v\|_{L^2}^2 + \frac{r}{1+t}\|(1+t)^r \nabla^k v\|_{L^2}^2\\
&\quad + C(1+t)^{2r} \sum_{0 \leq \ell \leq k-1} \intr |\nabla^{k-1-\ell} \rho| |\nabla^\ell (u-v)| |\nabla^{k+1}v|\,dx.
\end{align*}
On the other hand, we estimate
$$\begin{aligned}
&\lt(\intr |\nabla^{k-1-\ell} \rho|^2 |\nabla^\ell (u-v)|^2\,dx \rt)^{1/2} \cr
&\quad \leq \left\{ \begin{array}{ll}
 \|\nabla^{k-1-\ell} \rho\|_{L^\infty}\|\nabla^\ell (u-v)\|_{L^2} & \textrm{for $|k-1-\ell| \leq [d/2]$,}\\
  \|\nabla^{k-1-\ell} \rho\|_{L^2}\|\nabla^\ell (u-v)\|_{L^\infty} & \textrm{for $|k-1-\ell| \geq [d/2] + 1$}.
\end{array} \right.
\end{aligned}$$
This together with applying Lemma \ref{lem_moser} (ii) gives
$$\begin{aligned}
&C(1+t)^{2r} \sum_{0 \leq \ell \leq k-1} \intr |\nabla^{k-1-\ell} \rho| |\nabla^\ell (u-v)| |\nabla^{k+1}v|\,dx\cr
&\quad \leq C \sum_{1 \leq \ell \leq k} \|(1+t)^r\nabla^\ell (u-v)\|_{L^2}^2 + \frac12\|(1+t)^r\nabla^{k+1}v\|_{L^2}^2.
\end{aligned}$$
Hence we have
\[
\frac{d}{dt}\|(1+t)^r \nabla^k v\|_{L^2}^2 + \|(1+t)^r \nabla^{k+1}v\|_{L^2}^2 \le C\sum_{\ell=1}^k \lt(\|(1+t)^r \nabla^\ell u\|_{L^2}^2 +\|(1+t)^r \nabla^\ell v\|_{L^2}^2\rt),
\]
and our desired result directly follows from the above estimate.
\end{proof}

Then we combine two previous lemmas to yield the following uniform-in-time estimates.

\begin{corollary}\label{C3.1}
For $r \in (0,3/4)$, there exists a constant $C>0$ independent of $T$ such that 
\[\begin{aligned}
&\|(1+t)^r \nabla u(\cdot,t)\|_{H^s}^2 + \|(1+t)^r \nabla v(\cdot,t)\|_{H^{s-1}}^2 \cr
&\quad + \int_0^t \|(1+\tau)^r \nabla u(\cdot,\tau)\|_{H^s}^2\,d\tau + \int_0^t \|(1+\tau)^r \nabla v(\cdot,\tau)\|_{H^s}^2\,d\tau\\
&\qquad \le C\lt(\mathfrak{X}_0(s) + \|v_0\|_{L^1}^2\rt).
\end{aligned}\]
\end{corollary}
\begin{proof}
We proceed by induction to show that
\bq\label{C3-1.0}
\begin{aligned}
&\|(1+t)^r \nabla u(\cdot,t)\|_{H^\ell}^2 + \|(1+t)^r \nabla v(\cdot,t)\|_{H^{\ell-1}}^2 \cr
&\quad + \int_0^t \|(1+\tau)^r \nabla u(\cdot,\tau)\|_{H^\ell}^2\,d\tau+ \int_0^t \|(1+\tau)^r \nabla v(\cdot,\tau)\|_{H^\ell}^2\,d\tau\\
&\qquad \le C\lt(\mathfrak{X}_0(s) + \|v_0\|_{L^1}^2\rt)
\end{aligned}\eq
for $1 \le \ell \le s$, where $C>0$ is independent of $T$. Let us first show that \eqref{C3-1.0} holds with $\ell=1$. It follows from Proposition \ref{P3.1} that
\bq\label{C3-1.1}
\int_0^t \|(1+\tau)^r \nabla v(\cdot,\tau)\|_{L^2}^2\,d\tau \le C\lt(E(0)+\|v_0\|_{L^1}^2\rt) \le C\lt(\mathfrak{X}_0(s) + \|v_0\|_{L^1}^2\rt).
\eq
Then we combine \eqref{C3-1.1} with Lemma \ref{L3.3} with $k=1$ to obtain
\bq\label{C3-1.2}
\begin{aligned}
\|(1+t)^r \nabla u(\cdot,t)\|_{L^2}^2 + \int_0^t \|(1+\tau)^r \nabla u(\cdot,\tau)\|_{L^2}^2\,d\tau 
&\le C\lt(\|\nabla u_0\|_{L^2}^2+ \int_0^t \|(1+\tau)^r \nabla v(\cdot,\tau)\|_{L^2}^2\,d\tau\rt)\\
&\le C\lt( \mathfrak{X}_0(s) + \|v_0\|_{L^1}^2\rt).
\end{aligned}
\eq
As a consequence, we apply \eqref{C3-1.1} and \eqref{C3-1.2} to Lemma \ref{L3.4} with $k=1$ to yield
\bq\label{C3-1.3}\begin{aligned}
\|(&1+t)^r \nabla v(\cdot,t)\|_{L^2}^2 + \int_0^t \|(1+\tau)^r \nabla^2 v(\cdot,\tau)\|_{L^2}^2\,d\tau \\
& \le \|\nabla v_0\|_{L^2}^2 + C \lt(\int_0^t \|(1+\tau)^r \nabla u(\cdot,\tau)\|_{L^2}^2\,d\tau + \int_0^t \|(1+\tau)^r \nabla v(\cdot,\tau)\|_{L^2}^2\,d\tau\rt)\\
&\le C\lt(\mathfrak{X}_0(s) + \|v_0\|_{L^1}^2\rt).
\end{aligned}\eq
This again leads to
\bq\label{C3-1.4}
\begin{aligned}
&\|(1+t)^r \nabla^2 u(\cdot,t)\|_{L^2}^2 + \int_0^t \|(1+\tau)^r \nabla^2 u(\cdot,\tau)\|_{L^2}^2\,d\tau \cr
&\quad \le C\lt(\|\nabla^2 u_0\|_{L^2}^2+ \int_0^t \|(1+\tau)^r \nabla^2 v(\cdot,\tau)\|_{L^2}^2\,d\tau\rt)\cr
&\quad \le C\lt( \mathfrak{X}_0(s) + \|v_0\|_{L^1}^2\rt)
\end{aligned}
\eq
due to Lemma \ref{L3.3} with $k=2$.
Hence, we combine \eqref{C3-1.1}-\eqref{C3-1.4} to assert that \eqref{C3-1.0} holds when $\ell=1$.\\

\noindent Now, we assume that \eqref{C3-1.0} holds for some $s>\ell =m\ge1$. Then, Lemma \ref{L3.4} with $k=m+1 \leq s$ gives
\bq\label{C3-1.5}
\begin{aligned}
\|(&1+t)^r \nabla^{m+1} v(\cdot,t)\|_{L^2}^2 + \int_0^t \|(1+\tau)^r \nabla^{m+2} v(\cdot,\tau)\|_{L^2}^2\,d\tau \\
& \le \|\nabla^{m+1} v_0\|_{L^2}^2 + C\sum_{\ell=1}^{m+1} \lt(\int_0^t \|(1+\tau)^r \nabla^\ell u(\cdot,\tau)\|_{L^2}^2\,d\tau + \int_0^t \|(1+\tau)^r \nabla^\ell v(\cdot,\tau)\|_{L^2}^2\,d\tau\rt)\\
&\le C\lt(\mathfrak{X}_0(s) + \|v_0\|_{L^1}^2\rt).
\end{aligned}
\eq
Next, we apply \eqref{C3-1.5} to Lemma \ref{L3.3} with $k=m+2 \leq s+1$ to obtain
\bq\label{C3-1.6}
\begin{aligned}
\|&(1+t)^r \nabla^{m+2} u(\cdot,t)\|_{L^2}^2 + \int_0^t \|(1+\tau)^r \nabla^{m+2} u(\cdot,\tau)\|_{L^2}^2\,d\tau\\
& \le C\lt(\|\nabla^{m+2} u_0\|_{L^2}^2+ \int_0^t \|(1+\tau)^r \nabla^{m+2} v(\cdot,\tau)\|_{L^2}^2\,d\tau\rt)\\
&\le C\lt(\mathfrak{X}_0(s) + \|v_0\|_{L^1}^2\rt).
\end{aligned}
\eq
Thus, we combine \eqref{C3-1.5} and \eqref{C3-1.6} with the induction hypothesis to assert that \eqref{C3-1.0} holds when $\ell=m+1 \leq s$. This completes the inductive argument and gives the desired result.
\end{proof}

In the following two lemmas, we provide the estimates for $v$ and $u$ in $L^\infty(0,T;\dot H^{s+1}(\R^d)) \cap L^2(0,T;\dot H^{s+2}(\R^d))$ and $L^\infty(0,T;\dot H^{s+2}(\R^d)) \cap L^2(0,T;\dot H^{s+2}(\R^d))$, respectively, which are uniform in $T$. 
\begin{lemma}\label{L3.5}
There exists a constant $C>0$ independent of $T$ such that
\[
\|\nabla^{s+1} v(\cdot,t)\|_{L^2}^2 + \int_0^t \|\nabla^{s+2} v(\cdot,\tau)\|_{L^2}^2\,d\tau \le C\lt(\mathfrak{X}_0(s) + \|v_0\|_{L^1}^2\rt).
\]
\end{lemma}

\begin{proof}
Straightforward computation yields
\[\begin{aligned}
\frac12&\frac{d}{dt}\|\nabla^{s+1} v\|_{L^2}^2 + \|\nabla^{s+2} v\|_{L^2}^2 \\
&= -\intr \lt(\nabla^{s+1} (v\cdot \nabla v) - v \cdot \nabla(\nabla^{s+1} v)\rt) \nabla^{s+1} v\,dx -\intr \nabla^s (\rho(u-v))\nabla^{s+2} v\,dx\\
&\le C\|\nabla v\|_{L^\infty} \|\nabla^{s+1} v\|_{L^2}^2 + \|\nabla^s (\rho(u-v)\|_{L^2}\|\nabla^{s+2} v\|_{L^2}\\
&\le C \|\nabla^{s+1} v\|_{L^2}^2 + C\|\nabla^{s+2} v\|_{L^2} \lt( \|\rho\|_{L^\infty}\|\nabla^s(u-v)\|_{L^2} + \|\nabla^s \rho\|_{L^2}\|u-v\|_{L^\infty}\rt)\\
&\le C\|\nabla^{s+1} v\|_{L^2}^2 + \frac12\|\nabla^{s+2} v\|_{L^2}^2 + C\lt(\|u\|_{H^s}^2 + \|v\|_{H^s}^2\rt),
\end{aligned}\]
where we used Sobolev inequality and Young's inequality. Thus, we integrate the above relation with respect to $t$ and use Proposition \ref{P3.1}, Lemma \ref{L3.2} and Corollary \ref{C3.1} to get
\[\begin{aligned}
\|&\nabla^{s+1} v(\cdot,t)\|_{L^2}^2 + \int_0^t \|\nabla^{s+2} v(\cdot,\tau)\|_{L^2}^2\,d\tau\\
&\le \|\nabla^{s+1} v_0\|_{L^2}^2 + C\int_0^t \|\nabla^{s+1} v(\cdot,\tau)\|_{L^2}^2\,d\tau + C\int_0^t\lt( \|v(\cdot,\tau)\|_{L^2}^2 + \|u(\cdot,\tau)\|_{L^2}^2\rt) \,d\tau\\
&\quad + C\int_0^t \|\nabla v(\cdot,\tau)\|_{H^{s-1}}^2\,d\tau + C\int_0^t \|\nabla u(\cdot,\tau)\|_{H^{s-1}}^2\,d\tau\\
&\le C\lt(\mathfrak{X}_0(s) + \|v_0\|_{L^1}^2\rt) + C\int_0^t \|(1+\tau)^{1/2}\nabla v(\cdot,\tau)\|_{H^{s-1}}^2\,d\tau + C\int_0^t \|(1+\tau)^{1/2}\nabla u(\cdot,\tau)\|_{H^{s-1}}^2\,d\tau\\
&\le C\lt(\mathfrak{X}_0(s) + \|v_0\|_{L^1}^2\rt),
\end{aligned}\]
which implies the desired estimate.
\end{proof}

\begin{lemma}\label{L3.6}
There exists a constant $C>0$ independent of $T$ such that
\[
\|\nabla^{s+2} u(\cdot,t)\|_{L^2}^2 + \int_0^t \|\nabla^{s+2} u(\cdot,\tau)\|_{L^2}^2\,d\tau \le C(\mathfrak{X}_0(s) + \|v_0\|_{L^1}^2).
\]
\end{lemma}
\begin{proof}
Directly from the momentum equation in the Euler part of \eqref{A-1}, we obtain
\[\begin{aligned}
\frac12\frac{d}{dt}\|\nabla^{s+2} u\|_{L^2}^2 &= -\intr \lt(u \cdot \nabla (\nabla^{s+2} u)\rt) \nabla^{s+2} u \,dx -\intr \lt(\nabla^{s+2} (u\cdot \nabla u) - u \cdot \nabla(\nabla^{s+2} u)\rt) \nabla^{s+2} u\,dx\\
&\quad-\intr \nabla^{s+2} (u-v) \nabla^{s+2} u\,dx\\
&\le C\|\nabla u\|_{L^\infty}\|\nabla^{s+2} u\|_{L^2}^2  -\frac12\|\nabla^{s+2} u\|_{L^2}^2 + \|\nabla^{s+2} v\|_{L^2}^2\\
&\le -\lt(\frac12-C\e_1\rt) \|\nabla^{s+2} u\|_{L^2}^2  + \|\nabla^{s+2} v\|_{L^2}^2,
\end{aligned}\]
where $C$ is independent of $T$. Since $\e_1$ is sufficiently small, we can choose $1/2-C\e_1 >1/4$. Thus, we use Corollary \ref{C3.1} to get
$$\begin{aligned}
&\|\nabla^{s+2} u(\cdot,t)\|_{L^2}^2 +\int_0^t \|\nabla^{s+2} u(\cdot,\tau)\|_{L^2}^2\,d\tau \cr
&\quad \le \|\nabla^{s+2} u_0\|_{L^2}^2 + C\int_0^t \|\nabla^{s+2} v(\cdot,\tau)\|_{L^2}^2\,d\tau \le C\lt(\mathfrak{X}_0(s) + \|v_0\|_{L^1}^2\rt).
\end{aligned}$$
This concludes the desired result.
\end{proof}

We now estimate the fluid density $\rho$. We first show the uniform-in-time estimate on the upper bound for $\rho$, and then present the lower bound estimate. In particular, the lower bound estimate implies that the vacuum state cannot occur.
\begin{lemma}\label{C3.3}
There exists a constant $C>0$ independent of $T$ such that
\[
\|\rho(\cdot,t)\|_{H^s} \le \|\rho_0\|_{H^s} \exp\lt( C\lt(\sqrt{\mathfrak{X}_0(s)} + \|v_0\|_{L^1}\rt)\rt).
\]
\end{lemma}
\begin{proof}
We first start with the $L^2$-estimate of $\rho$. Straightforward computation gives
\[
\frac12\frac{d}{dt}\|\rho\|_{L^2}^2 = -\intr\nabla\cdot (\rho u) \rho\,dx =\frac12 \intr u \cdot \nabla|\rho|^2\,dx \le \frac12\|\nabla\cdot u\|_{L^\infty}\|\rho\|_{L^2}^2.
\]
For $ 1\le k \le s$, we use Lemma \ref{lem_moser} to obtain
\begin{align*}
\frac12\frac{d}{dt}\|\nabla^k \rho\|_{L^2}^2 &= -\intr \nabla(\nabla^k \rho)\cdot u \ \nabla^k\rho\,dx -\intr \lt(\nabla^k (\nabla \rho \cdot u ) - \nabla (\nabla^k \rho)\cdot u)\rt)\nabla^k\rho\,dx\\
&\quad -\intr \rho \nabla^k (\nabla \cdot u) \nabla^k\rho\,dx -\intr \lt(\nabla^k(\rho \nabla \cdot u) -\rho\nabla^k(\nabla\cdot u)\rt)\nabla^k \rho\,dx\\
&\le \frac12\|\nabla u\|_{L^\infty}\|\nabla^k \rho\|_{L^2}^2 + C\|\nabla^k (\nabla \rho \cdot u ) - \nabla (\nabla^k \rho)\cdot u)\|_{L^2}\|\nabla^k \rho\|_{L^2}\cr
&\quad + \|\rho\|_{L^\infty}\|\nabla^{k+1} u\|_{L^2}\|\nabla^k \rho\|_{L^2} + C\|\nabla^k(\rho \nabla \cdot u) -\rho\nabla^k(\nabla\cdot u)\|_{L^2}\|\nabla^k \rho\|_{L^2}\cr
&\leq C\|\nabla u\|_{L^\infty}\|\nabla^k \rho\|_{L^2}^2 + C\|\nabla \rho\|_{L^\infty}\|\nabla^k u\|_{L^2}\|\nabla^k \rho\|_{L^2}+ \|\rho\|_{L^\infty}\|\nabla^{k+1} u\|_{L^2}\|\nabla^k \rho\|_{L^2}\cr
&\le C\|\nabla u\|_{H^s}\|\rho\|_{H^s}^2,
\end{align*}
where $C$ is independent of $T$ and we used
\[
\|\rho\|_{W^{1,\infty}} \lesssim  \|\rho\|_{H^{[d/2]+2}} \lesssim \|\rho\|_{H^{s - ([d/2] -1)}} \lesssim \|\rho\|_{H^s}
\]
due to Lemma \ref{lem_moser} (ii). 

Now, we gather all the results to yield
\[
\frac{d}{dt}\|\rho\|_{H^s}^2 \le C\|\nabla u\|_{H^s}\|\rho\|_{H^s}^2,
\]
where $C$ is independent of $T$. We use Gr\"onwall's lemma and Corollary \ref{C3.1} to obtain, for some $\alpha \in (1/2,3/4)$,
\[\begin{aligned}
\|\rho(\cdot,t)\|_{H^s} &\le \|\rho_0\|_{H^s}\exp\lt( C\int_0^t \|\nabla u(\cdot,\tau)\|_{H^s}\,d\tau\rt)\\
&\le \|\rho_0\|_{H^s}\exp\lt(C \lt(\int_0^t \|(1+\tau)^\alpha \nabla u(\cdot,\tau)\|_{H^s}^2\,d\tau\rt)^{1/2} \lt(\int_0^t (1+\tau)^{-2\alpha}\,d\tau\rt)^{1/2}\rt)\\
&\le \|\rho_0\|_{H^s} \exp\lt( C\lt(\sqrt{\mathfrak{X}_0(s)} + \|v_0\|_{L^1}\rt)\rt),
\end{aligned}\]
which gives the desired result.
\end{proof}
For the lower bound estimate of $\rho$, we define a backward characteristic flow $\eta = \eta(x,t)$ by
\[
\pa_s \eta(x,t) = u(\eta(x,s),s) \quad \mbox{with} \quad \eta(x,t) = x.
\]
Note that $\eta$ is well-defined due to the strong regularity on $u$.
\begin{lemma} There exists a constant $C>0$ independent of $T$ such that
\[
\rho(x,t) \geq \rho_0(\eta(x,0))\exp\lt(-C\lt(\sqrt{\mathfrak{X}_0(s)} + \|v_0\|_{L^1}\rt) \rt) > 0.
\]
\end{lemma}
\begin{proof}It follows from the continuity equation in \eqref{A-1} that
\[
\pa_s \rho(\eta(x,s),s) = -(\nabla \cdot u)(\eta(x,s),s)\rho(\eta(x,s),s),
\]
and this gives
\[
\rho(x,t) = \rho_0(\eta(x,0))\exp\lt( -\int_0^t(\nabla \cdot u)(\eta(x,\tau),\tau) \,d\tau\rt).
\]
On the other hand, similarly as in the proof of Lemma \ref{C3.3}, we estimate
\[
\lt|\int_0^t(\nabla \cdot u)(\eta(x,\tau),\tau) \,d\tau\rt| \leq \int_0^t \|\nabla u(\cdot,\tau)\|_{L^\infty}\,d\tau \leq C\int_0^t \|\nabla u(\cdot,\tau)\|_{H^s}\,d\tau \leq C\lt(\sqrt{\mathfrak{X}_0(s)} + \|v_0\|_{L^1}\rt).
\]
Hence we have
\[
\rho(x,t) \geq \rho_0(\eta(x,0))\exp\lt(-C\lt(\sqrt{\mathfrak{X}_0(s)} + \|v_0\|_{L^1}\rt) \rt),
\]
where $C>0$ is independent of $T$.
\end{proof}
Finally, we combine all the previous results to prove Proposition \ref{P3.2}
\begin{proof}[Proof of Proposition \ref{P3.2}]
Lemma \ref{C3.3} directly implies
\bq\label{P3-2.1}
\sup_{0\le t\le T} \|\rho(\cdot,t)\|_{H^s}^2 \le C\|\rho_0\|_{H^s}^2 \le C\lt(\mathfrak{X}_0(s) + \|v_0\|_{L^1}^2\rt).
\eq
Next, Lemma \ref{L3.2}, Corollary \ref{C3.1}, and Lemma \ref{L3.5} yield
\bq\label{P3-2.2}
\begin{aligned}
\sup_{0\le t \le T} \|u(\cdot,t)\|_{H^{s+2}}^2 &\le \sup_{0\le t \le T}\lt( \|u(\cdot,t)\|_{L^2}^2 + \|\nabla^{s+2} u(\cdot,t)\|_{L^2}^2 + \|\nabla u(\cdot,t)\|_{H^s}^2\rt)\\
&\le \sup_{0\le t \le T}\lt( \|u(\cdot,t)\|_{L^2}^2 + \|\nabla^{s+2} u(\cdot,t)\|_{L^2}^2 + \|(1+t)^{1/2}\nabla u(\cdot,t)\|_{H^s}^2\rt)\\
&\le C\lt(\mathfrak{X}_0(s) + \|v_0\|_{L^1}^2\rt).
\end{aligned}
\eq
Similarly, we use Proposition \ref{P3.1}, Lemma \ref{L3.6}, and Corollary \ref{C3.1} to get
\bq\label{P3-2.3}
\sup_{0\le t\le T} \|v(\cdot,t)\|_{H^{s+1}}^2 \le C\lt(\mathfrak{X}_0(s) + \|v_0\|_{L^1}^2\rt).
\eq
Thus, we collect \eqref{P3-2.1}, \eqref{P3-2.2} and \eqref{P3-2.3} to conclude the desired result.
\end{proof}

Now, we proceed to the proof of Theorem \ref{T2.1} in the following two subsections.

\subsection{Global-in-time existence of classical solutions} Let us first take into account the global-in-time existence part in Theorem \ref{T2.1}. We choose a positive constant $\e_1\ll1$ sufficiently small so that it satisfies the required smallness condition in Lemma \ref{L3.2} and Lemma \ref{L3.3}. Then, assume that
\[
\mathfrak{X}_0(s) + \|v_0\|_{L^1}^2 \le \frac{\e_1^2}{2(1+C^*)},
\]
where $C^* > 0$ appeared in Proposition \ref{P3.2}. Then we set
\[
\mathcal{S} := \lt\{ T\ge 0 \ | \ \mathfrak{X}(s;T) <\e_1^2\rt\}. 
\]
By the local-in-time existence theorem in Theorem \ref{thm_local}, the set $\mathcal{S}$ is non-empty. Now, we argue by contradiction to show $\sup\mathcal{S} = \infty$. Assume that $T^*:= \sup\mathcal{S} <\infty$. Then we have
\[
\e_1^2 = \lim_{t \to T*-}\mathfrak{X}(s;t) \le C^*\lt(\mathfrak{X}_0(s) + \|v_0\|_{L^1}^2\rt)\le \frac{C^*}{2(1+C^*)}\e_1^2 \le \frac{\e_1^2}{2},
\]
which leads to a contradiction. This implies $T^* = \infty$, and hence the classical solution obtained in Theorem \ref{thm_local} globally exists.

\subsection{Large-time behavior of solutions} In this part, we provide the details on the large-time behavior estimate in Theorem \ref{T2.1}. We separately consider zeroth-order and higher-order estimates as follows.\\

\noindent $\diamond$ (Zeroth-order estimates)  Since we now have $\|\rho\|_{L^\infty(\R^d \times \R_+)} <\infty$, we can repeat the procedure in Proposition \ref{P3.1} so that the relation \eqref{large_time} actually holds for all $t\ge 0$. Thus we obtain 
\[
\intr \rho|u|^2(x,t)\,dx + \intr |v(x,t)|^2\,dx \le \frac{C}{(1+t)^\alpha} \quad \forall\, t \ge 0
\]
for every $\alpha \in (0,d/2)$. For the $L^2$-norm of $u$, we use the above decay estimate and Lemma \ref{L3.2} to find
\[
\frac{d}{dt}\|u\|_{L^2}^2 + \|u\|_{L^2}^2 \le 2\|v\|_{L^2}^2 \le  \frac{C}{(1+t)^\alpha}.
\]
Then applying the Gr\"onwall's lemma gives
\[
\|u(\cdot,t)\|_{L^2}^2 \leq \|u_0\|_{L^2}^2e^{-t} + Ce^{-t} \int_0^t \frac{e^s}{(1+s)^\alpha}\,ds.
\]
On the other hand, we estimate
$$\begin{aligned}
\int_0^t \frac{e^s}{(1+s)^\alpha}\,ds &= \int_0^{t/2} \frac{e^s}{(1+s)^\alpha}\,ds + \int_{t/2}^t \frac{e^s}{(1+s)^\alpha}\,ds\cr
&\leq \int_0^{t/2} e^s\,ds + \frac{1}{(1 + t/2)^\alpha}\int_{t/2}^t e^s\,ds\cr
&=e^{t/2} - 1 + \frac{e^t - e^{t/2}}{(1 + t/2)^\alpha}.
\end{aligned}$$
Thus we have
$$\begin{aligned}
\|u(\cdot,t)\|_{L^2}^2 &\leq \|u_0\|_{L^2}^2e^{-t} + C\lt(e^{-t/2} - e^{-t}  \rt) + C\frac{1 - e^{-t/2}}{(1 + t/2)^\alpha} \leq \frac{C}{(1+t)^\alpha} \quad \forall\,t \geq 0
\end{aligned}$$
for some $C>0$ independent of $t$. \newline

\noindent $\diamond$ (Higher-order estimates) If $d=3$, then Corollary \ref{C3.1} directly implies
\[
\|\nabla u(\cdot,t)\|_{H^s}^2 + \|v(\cdot,t)\|_{H^s}^2 \le \frac{C}{(1+t)^\alpha} \quad \forall \, t \ge 0
\]
for $\alpha \in (0,3/2)$. On the other hand, if $d \ge 4$, it is required to reinvestigate the arguments in Lemma \ref{L3.3}. In this case, we choose $r \in (0,d/4)$ and follow the argument in the proof of Lemma \ref{L3.3} to deduce
\begin{align*}
&\frac12\frac{d}{dt}\|(1+t)^r \nabla^k u\|_{L^2}^2 \cr
&\quad \le -(1-C\e_1) \|(1+t)^r \nabla^k u\|_{L^2}^2 + \|(1+t)^r \nabla^k u\|_{L^2} \|(1+t)^r \nabla^k v\|_{L^2} + \frac{r}{1+t} \|(1+t)^r\nabla^k u\|_{L^2}^2
\end{align*}
for $1\le k \le s+1$. In order to control the last term on the right hand side of the above inequality, we need to consider a large time. More precisely, if $t \ge T_d:= d/3-1$, then we get
\[
\frac{r}{1+t} < \frac{d}{4(1+T_d)} = \frac34.
\]
This yields
\begin{align*}
\frac12\frac{d}{dt}\|(1+t)^r \nabla^k u\|_{L^2}^2 &\le -(1/4-C\e_1) \|(1+t)^r \nabla^k u\|_{L^2}^2 + \|(1+t)^r \nabla^k u\|_{L^2} \|(1+t)^r \nabla^k v\|_{L^2}\\
&\le -(1/8-C\e_1) \|(1+t)^r \nabla^k u\|_{L^2}^2 + 2 \|(1+t)^r \nabla^k v\|_{L^2}^2\\
&\le -\frac{1}{16}\|(1+t)^r \nabla^k u\|_{L^2}^2 + 2 \|(1+t)^r \nabla^k v\|_{L^2}^2,
\end{align*}
where we used the smallness of solutions to have $1/8-C\e_1 > 1/16$. Thus, by integrating the above inequality over the time interval $[T_d,t]$, we obtain
\[
\begin{aligned}
\|&(1+t)^r \nabla^k u(\cdot,t)\|_{L^2}^2 + \int_{T_d}^t \|(1+\tau)^r \nabla^k u(\cdot,\tau)\|_{L^2}^2\,d\tau \\
&\le \|(1+\tau)^r \nabla^k u(\cdot,\tau)\|_{L^2}^2\bigg|_{\tau = T_d} + C \int_{T_d}^t  \|(1+\tau)^r \nabla^k v(\cdot,\tau)\|_{L^2}^2\,d\tau\\
&\le C\lt(\mathfrak{X}_0(s) + \|v_0\|_{L^1}^2\rt)+ C \int_{T_d}^t  \|(1+\tau)^r \nabla^k v(\cdot,\tau)\|_{L^2}^2\,d\tau.
\end{aligned}
\]
Similarly, we can also derive the following inequality by using the argument in the proof of Lemma \ref{L3.4}: 
\[
\begin{aligned}
\|&(1+t)^r \nabla^k v(\cdot,t)\|_{L^2}^2 + \int_{T_d}^t \|(1+\tau)^r \nabla^{k+1} v(\cdot,\tau)\|_{L^2}^2\,d\tau \\
& \le \|(1+\tau)^r \nabla^k v(\cdot,\tau)\|_{L^2}^2 \bigg|_{\tau = T_d}+ C\sum_{\ell=1}^k \lt(\int_{T_d}^t \|(1+\tau)^r \nabla^\ell u(\cdot,\tau)\|_{L^2}^2\,d\tau + \int_{T_d}^t \|(1+\tau)^r \nabla^\ell v(\cdot,\tau)\|_{L^2}^2\,d\tau\rt)
\end{aligned}
\]
for $t \ge T_d$, $1\le k \le s$ and  $r \in (0, d/4)$. On the other hand, Proposition \ref{P3.1} implies
\[
\int_{T_d}^t \|(1+\tau)^r \nabla v(\cdot,\tau)\|_{L^2}^2\,d\tau \le  \int_0^t \|(1+\tau)^r \nabla v(\cdot,\tau)\|_{L^2}^2\,d\tau\le C\lt(\mathfrak{X}_0(s) + \|v_0\|_{L^1}^2\rt)
\]
for $t \ge T_d$ and $r \in (0,d/4)$. We finally proceed by induction as Corollary \ref{C3.1} to get
\[\begin{aligned}
&\|(1+t)^r \nabla u(\cdot,t)\|_{H^s}^2 + \|(1+t)^r \nabla v(\cdot,t)\|_{H^{s-1}}^2 \\
& \quad + \int_{T_d}^t \|(1+\tau)^r \nabla u(\cdot,\tau)\|_{H^s}^2\,d\tau + \int_{T_d}^t \|(1+\tau)^r \nabla v(\cdot,\tau)\|_{H^s}^2\,d\tau\\
& \qquad \le C\lt(\mathfrak{X}_0(s) + \|v_0\|_{L^1}^2\rt)
\end{aligned}\]
for $r \in (0,d/4)$ and $t \ge T_d$. For $t \leq T_d$, the solutions $(u,v)$ are bounded in $\mc([0,T_d];H^{s+1}(\R^d)) \times \mc([0,T_d];H^s(\R^d))$, thus we can find a constant $C$ depending on $T_d$ such that 
\[
\|(1+t)^r \nabla u(\cdot,t)\|_{H^s}^2 + \|(1+t)^r \nabla v(\cdot,t)\|_{H^{s-1}}^2 \leq C\lt(\mathfrak{X}_0(s) + \|v_0\|_{L^1}^2\rt)
\]
for all $t \in [0,T_d]$. Combining the above two estimates completes the proof.

\section*{Acknowledgments}

The work of Y.-P. Choi was supported by NRF grant (No. 2017R1C1B2012918), POSCO Science Fellowship of POSCO TJ Park Foundation, and Yonsei University Research Fund of 2019-22-021. The work of J. Jung was supported by NRF grant (No. 2019R1A6A1A10073437).

\end{document}